\documentclass[article, 12pt]{amsart}

\usepackage{amsmath}
\usepackage{amssymb}
\usepackage{mathrsfs}
\usepackage{verbatim}
\usepackage{enumitem}
\usepackage{svn}
\usepackage{stmaryrd}
\usepackage[all]{xy}
\usepackage{amsfonts}
\usepackage{comment}
\usepackage{tabu} 
  
\usepackage[table]{xcolor}
\usepackage{adjustbox}
\usepackage{color}
\definecolor{ghcolor}{RGB}{0, 150, 200} 
\definecolor{winestain}{rgb}{0.5,0,0}
\usepackage[linktocpage=true, hidelinks, colorlinks, linkcolor=blue, citecolor=red]{hyperref} 

\usepackage[a4paper,bindingoffset=0.2in,left=0.7in,right=0.7in,top=0.7in,bottom=0.7in, footskip=.25in]{geometry}

\swapnumbers

\newtheorem{theorem}[subsubsection]{Theorem}
\newtheorem{thm}[subsubsection]{Theorem}
\newtheorem{lemma}[subsubsection]{Lemma}

\newtheorem{cor}[subsubsection]{Corollary}

\newtheorem{prop}[subsubsection]{Proposition}

\theoremstyle{definition}
\newtheorem{defn}[subsubsection]{Definition}
\newtheorem{definition}[subsubsection]{Definition}
\newtheorem{example}[subsubsection]{Example}

\newtheorem{convention}[subsubsection]{Convention}

\theoremstyle{remark}
\newtheorem{remark}[subsubsection]{Remark}
\newtheorem{rem}[subsubsection]{Remark}

\numberwithin{equation}{subsection}

\def\numequation{\addtocounter{subsubsection}{1}\begin{equation}}
\def\nummultline{\addtocounter{subsubsection}{1}\begin{multline}}

\def \m {\mathfrak M}

 \def \E{\mathcal E}

\def \Z {\mathbb Z}
\def \inj {\hookrightarrow }
\def \to {\rightarrow}
\def \onto {\twoheadrightarrow}
\def \spec \text{spec}
\def \cont \text{cont}
 \def \M{\mathfrak M}

\def \GL {\t{GL}}
\def \L {\mathfrak L}

\DeclareMathOperator{\gal}{Gal}

\def \Q {\mathbb Q}

\def \t {\textnormal}
\def \Z {\mathbb Z}

\def \O {\mathcal O}

\def \BF {\mathbb F}
\def \gs {\mathfrak S}

\def \ur {\t{ur}}

\def \m {\mathfrak M}
\def \Ker {\textnormal{Ker}}

\def \ku {k \llbracket u\rrbracket}

\def \Fr {\t{Fr}}

\def \acris {{A_{\t{cris}}}}

\def \sfi {\t{Mod}_{\gs}^{\varphi, r}}

\def \st {\textnormal{st}}

\def \< {\left <}
\def \> {\right >}

\def \hM {{\hat \M}}

\def \gr{\textnormal{gr}}

\def \ue {{\underline{\varepsilon}}}

\def \Md {{\rm M}_d}

\def \fe {{\mathfrak e}}
\def \e {{\mathfrak e}}

\def \Zp { \mathbb Z_p}
\def \Qp { \mathbb Q_p}
\def \Res {\textnormal{Res}}

\def \Fp  {\BF_p}

\def \FrR {\t{Fr} R}

\def \sto {\twoheadrightarrow}

\DeclareMathOperator{\Gal}{Gal}

\DeclareMathOperator{\Ind}{Ind}

\newcommand*{\wt}[1]{\widetilde{#1}}

\begin{document}

\title{Wach models and overconvergence of \'etale $(\varphi, \Gamma)$-modules}
\date{\today}
\author{Hui Gao}
\address{Department of Mathematics and Statistics, University of Helsinki, FI-00014, Finland}
\email{hui.gao@helsinki.fi}
\thanks{The author is partially supported by a postdoctoral position funded by Academy of Finland through Kari Vilonen.}



\begin{abstract}
A classical result of Cherbonnier and Colmez says that all \'etale $(\varphi, \Gamma)$-modules are overconvergent. In this paper, we give another proof of this fact when the base field $K$ is a finite extension of $\mathbb Q_p$.
Furthermore, we obtain an explicit (``uniform") lower bound for the overconvergence radius, which was previously not known.
The method is similar to that in a previous joint paper with Tong Liu. Namely, we study Wach models (when $K$ is unramified) in modulo $p^n$ Galois representations, and
use them to build an overconvergence basis.
\end{abstract}


\maketitle

\tableofcontents

\section{Introduction}

\subsection{The classical theorem of Cherbonnier-Colmez} \label{subsub K}
Let us first set up some notations.
Let  $p$ be a prime. Let $k$ be a perfect field of characteristic $p$, $W(k)$ its ring of Witt vectors, $K_0:=W(k)[1/p]$, and $K/K_0$ a totally ramified finite extension with $e$ the ramification index. We fix an algebraic closure $\overline {K}$ of $K$ and set $G_K:=\Gal(\overline{K}/K)$.
Define $\mu _n \in \overline K$ inductively such that $\mu_1$ is a primitive $p$-th root of unity and $(\mu_{n+1})^p = \mu_n$, and let $K_{p^\infty}:=  \cup _{n=1}^\infty K(\mu_{n})$. Let $H_K := \gal (\overline K / K_{p^\infty})$, and $\Gamma_K: =\gal(K_{p^\infty}/K)$.

Let $T$ be a finite free $\Zp$-representation of $G_K$ of rank $d$.
In \cite{fo4}, for each such $T$ is associated an \'etale $(\varphi, \Gamma)$-module $\underline{\hat{M}}(T)$ where
$$\underline{\hat{M}}(T) :=  (\O _{\widehat \E ^\ur} \otimes_{\Z_p} T) ^{H_K}.$$
See \S \ref{sec phi gamma} for any unfamiliar terms and more details. Here $ \O_{\widehat \E^\ur}$ is a certain ``period ring".
For the $\Zp$-representation $T$, we can also define its ``overconvergent periods" via:
$$ \underline{\hat{M}}^{\dagger, r} (T): =(\O _{\widehat \E ^\ur}^{\dagger, r} \otimes_{\Z_p} V) ^{H_K}.$$
where $r \in \mathbb R^{>0}$, and $\O _{\widehat \E ^\ur}^{\dagger, r}$ is the ``overconvergent period ring".
We say that $\hat M(T)$ is \emph{overconvergent} if it can be recovered by its ``overconvergent periods", i.e., if for some $r \in \mathbb R^{>0}$, we have
\begin{equation} \label{eq def}
\underline{\hat{M}}(T) = \O_\E  \otimes_{\O_\E ^{\dag, r}}   \underline{\hat{M}}^{\dagger, r} (T).
\end{equation}
The main theorem of \cite{Colmez-overcon} is the following:


\begin{thm}[\cite{Colmez-overcon}]\label{thm intro OC}
For any finite free $\Z_p$-representation $T$ of $G_K$, its associated $(\varphi, \Gamma)$-module is overconvergent.
\end{thm}

\begin{rem}
The theorem was also reproved (and generalized to family version) by Berger and Colmez \cite{BC08}. There is also a relatively more direct proof by Kedlaya \cite{Kednew}.
\end{rem}

\subsection{A reproof when $K/\Qp$ is finite}
In this paper, we give another proof of Theorem \ref{thm intro OC} when $K$ is a finite extension of $\mathbb Q_p$. Furthermore, we obtain an explicit ``uniform" lower bound (depending only on $p, f, e, d$ where $f:=[k: \Fp]$) on the overconvergence radius, which was previously not known.

\begin{thm} \label{thm1}
Suppose $K/\Qp$ is a finite extension.
Let $T$ be a finite free $\Zp$-representation of $G_K$ of rank $d$.
Then $\underline{\hat{M}}(T)$ is overconvergent on the interval
\begin{equation} \label{eqeq radius1}
(0, \frac{1}{21pfe^2d^2p^{fed}}],
\end{equation}
i.e., \eqref{eq def} holds for $r= {1}/{(21pfe^2d^2p^{fed})}$ (see \S \ref{subsec OC} for our conventions).
\end{thm}

\begin{rem}
\begin{enumerate}
\item We need $[K: \Qp] < \infty$ in order to apply the loose crystalline lifting results in \cite{GL}, see \cite[Rem. 1.1.2(1)]{GL} for some remarks on this condition.

\item There are some results concerning the overconvergence radius in \cite[\S 4.2]{BC08} (cf. \cite[Lem. 4.2.5, Prop. 4.2.6]{BC08}). We will show that using \textit{loc. cit.}, together with results from \cite{GL} and \cite{liu-car2}, we can also prove a certain ``uniform overconvergence" (but only \emph{implicitly}); see \S \ref{sub compa} for more details.
\end{enumerate}
\end{rem}

\begin{rem}
It is interesting to mention that in \cite{GP}, another proof of the overconvergence property of \'etale $(\varphi, \tau)$-modules is given (which works without assuming $[K: \Qp]<\infty$); however, the proof makes full use of Thm. \ref{thm intro OC}.
\end{rem}

The overconvergence theorem \ref{thm intro OC} plays a fundamental role in the application of $(\varphi, \Gamma)$-modules to study various problems, see, e.g., \cite[\S 1.2]{GL} for some discussion.
In particular, as we mentioned in \textit{loc. cit}, overconvergence helps to link the category of all Galois representations to the category of geometric (i.e., semi-stable, crystalline) representations.
Similarly as in \cite{GL}, we will already use such a link to prove Theorem \ref{thm1}. Namely, we use crystalline representations to ``approximate" general Galois representations.
By writing this paper, we hope that our approach can shed some more light on the deeper meaning of overconvergence. We also hope that this paper can serve as a useful companion to the paper \cite{GL}, so that the readers can compare between the setting of $(\varphi, \Gamma)$-modules and $(\varphi, \tau)$-modules.

\subsection{Strategy of proof}
As we mentioned earlier, the strategy is similar to \cite{GL}, with one particular caveat. Namely, for a lattice in a crystalline representation of $G_K$, we can always attach a $(\varphi, \hat G)$-module by the main result of \cite{liu4}, but we can not always attach a Wach module. Fortunately, when $K$ is unramified, i.e., when $K=K_0$, we can always attach a Wach module (we have to be careful here to avoid a ``circular reasoning", see Remark \ref{rem Wach citing} for more details).
Now, by an argument in \cite{Colmez-overcon}, overconvergence of Galois representations is insensitive to inductions (of representations); thus it suffices to prove Theorem \ref{thm1} when $K=K_0$, where we always have Wach modules (for crystalline representations).
Then the strategy (indeed, the proof itself) will be parallel to \cite{GL}, except a few minor changes (in particular, the $\varphi$-action on period rings, cf. \S \ref{item comparison}).


Let us give a quick sketch of the strategy here (very similar to \cite[\S 1.3]{GL}). We assume $K/\Qp$ is finite and $K=K_0$ in this paragraph.
By the loose crystalline lifting theorem in \cite{GL}, we can easily deduce that for each $n \ge 1$, $T_n: =T/p^nT$ admits a unique \emph{maximal liftable Wach model}. Then we can analyse these models, and use them to build an overconvergence basis to prove the theorem.

\subsection{Structure of the paper}
In \S \ref{sec phi gamma}, we collect basic facts about \'etale $(\varphi, \Gamma)$-modules and Wach modules. We define what it means for an \'etale $(\varphi, \Gamma)$-module to be overconvergent, and state the theorem of Cherbonnier-Colmez.
In \S \ref{sec Wach}, when $K/\Qp$ is finite and $K=K_0$, we show the existence of Wach models, and analyse their various properties.
Finally in \S \ref{sec OC}, we build an overconvergence basis to prove the main theorem; we also make some comparison with known proofs of overconvergence.

\subsection{Notations}
\subsubsection{Some notations in $p$-adic Hodge theory}
We put $R:=\varprojlim \limits_{x\to x^p} \O_{\overline K}/ p \O_{\overline K}$, equipped with its natural coordinate-wise action of $G_K$. Let $\Fr R$ be its fraction field.
 There is a unique surjective projection map $\theta :
W(R) \to \widehat \O_{\overline K}$ to the $p$-adic completion $ \widehat \O_{\overline K}$ of
$\O_{\overline K}$, which lifts the  projection $R \to \O_{\overline K}/ p$
onto the first factor in the inverse limit.
Let  $A_{\t {cris}}$, $B_\t{cris}^+$, $B_\t{dR}^+$ be the usual period rings.


Recall that we defined $\mu _n \in \overline K$ inductively such that $\mu_1$ is a primitive $p$-th root of unity and $(\mu_{n+1})^p = \mu_n$.
Set  $\underline \varepsilon:= (\mu_{i}) _{i \geq 0} \in R$. Let $[\underline \varepsilon]\in W(R)$ be the Techm\"uller representative, and $t:= -\log([\underline \varepsilon])\in \acris$ as usual.
Let $\gs  = W[\![u]\!]$ with Frobenius extending the arithmetic Frobenius on $W(k)$ and $\varphi (u ) = (u+1)^p-1$.
We can embed the $W(k)$-algebra $W(k)[u]$ into
$W(R)\subset\acris$ by the map $u\mapsto 1-[\underline \varepsilon]$.   This
embedding extends to an  embedding $\gs \inj W(R)$ which is compatible with
Frobenious endomorphisms.

\subsubsection{Fontaine modules and Hodge-Tate weights}
When $V$ is a semi-stable representation of $G_K$, we let $D_{\textnormal{st}} (V):  = (B_{\st} \otimes_{\Qp} V^{\vee})$ where $V^{\vee}$ is the dual representation of $V$. The Hodge-Tate weights of $V$ are defined to be $i \in \mathbb Z$ such that $\gr^i D_{\st}(V) \neq 0$.  For example, for the cyclotomic character $\varepsilon_p$, its Hodge-Tate weight is $\{ 1\}$.

\subsubsection{Some other notations}
Throughout this paper, we reserve $\varphi$ to denote Frobenius operator. We sometimes add subscripts to indicate on which object Frobenius is defined. For example, $\varphi_\M$ is the Frobenius defined on $\M$. We always drop these subscripts if no confusion will
arise. Let $S$ be a ring endowed with Frobenius $\varphi_S$ and  $M$  a module over $S$. We always denote $\varphi ^*M := S \otimes_{ \varphi_S, S } M$. Note that if $M$ has a $\varphi _S$-semi-linear endomorphism $\varphi_M: M \to M$ then $1 \otimes \varphi_M : \varphi ^* M \to M$ is an $S$-linear map. 
We also reserve $v$ to denote valuations which is normalized so that $v(p)=1$.
Finally $\Md (S)$ always denotes the ring of $d \times d$-matrices with entries in $S$ and $I_d$ denotes the $d \times d $-identity matrix.

\textbf{Acknowledgement.} I thank Laurent Berger and Tong Liu for some useful discussions. I also thank the anonymous referee(s) for helping to improve the exposition.

\section{$(\varphi, \Gamma)$-modules, Wach modules, and overconvergence} \label{sec phi gamma}
In this section, except in \S \ref{subsec reproof} (the final subsection), we will let $K$ be as in \S \ref{subsub K}, i.e., $K/\Qp$ is not necessarily finite.

In this section, we first collect some basic facts on (integral and torsion) \'etale $\varphi$-modules, \'etale $(\varphi, \Gamma)$-modules, Wach modules and their attached representations.
Then, we define what it means for an \'etale $(\varphi, \Gamma)$-module to be overconvergent, and state the classical overconvergence theorem.

\subsection{\'Etale $\varphi$-modules and $(\varphi, \Gamma)$-modules}
Let $\O _\E$ be the $p$-adic completion of $\gs[1/u]$.  
Our fixed embedding $\gs\hookrightarrow W(R)$ determined by  $\ue$
uniquely extends to a $\varphi$-equivariant embedding $\iota:\O_{\E}\hookrightarrow W(\t{Fr} R)$, and we identify $\O_{\E}$ with its image in $W(\t{Fr} R)$.
We note that $\O_\E$ is a complete discrete valuation ring with uniformizer $p$ and residue field
$k (\!(\ue)\!)$ as a subfield of $\Fr R$. Let $\E$ denote the fractional field of $\O_\E$,  $\E ^\ur$ the maximal unramified extension of $\E$ inside $W(\t{Fr} R )[\frac 1 p ]$ and $\O_{\E ^\ur}$ the ring of integers. Set $\O_{\widehat{\E} ^\ur}$ the $p$-adic completion of $\O_{\E ^\ur}$ and  $\gs ^\ur : = W(R) \cap \O _{\widehat \E^\ur}$.

Let $\O_{\E_K} := (\O_{\widehat{\E} ^\ur})^{H_K}$, and so $\Gamma_K$ acts on it.
When $K=K_0$ (i.e., $K$ is unramified), we actually have $\O_{\E_K} =\O_\E$ (see e.g., first paragraph of \cite[\S I.2]{Ber}); in this case, we have $\gamma(u)=(1+u)^{\varepsilon_p(\gamma)}-1$ for $\gamma \in \Gamma_K$.

\begin{defn}
Let $'\t{Mod}_{\O_{\E_K}}^\varphi$ denote the category of finite type $\O_{\E_K}$-modules $M $  equipped with a $\varphi_{\O_{\E_K}}$-semi-linear endomorphism $\varphi _M : M\to M$ such that $1 \otimes \varphi : \varphi ^*M \to M $ is an isomorphism. Morphisms in this category  are just $\O_{\E_K}$-linear maps compatible with $\varphi$'s. We call objects in $'\t{Mod}_{\O_{\E_K}}^\varphi$ {\em \'etale $\varphi$-modules}.
\end{defn}

\subsubsection{} Let $'\t{Rep}_{\Z_p}(H_K) $ (resp. $'\t{Rep}_{\Z_p}(G_{K}) $ ) denote the category of finite type $\Z_p$-modules $T$   with a continuous $\Z_p$-linear $H_K$ (resp. $G_K$)-action.
For $M $ in $'\t{Mod}_{\O_{\E_K}}^\varphi$, define
$$ V(M):= ( \O _{\widehat \E ^\ur} \otimes_{\O_{\E_K}} M) ^{\varphi =1}.  $$
For $V $ in  $'\t{Rep}_{\Z_p}(H_K) $, define
$$ \underline M(V):= ( \O _{\widehat \E ^\ur} \otimes_{\Z_p} V) ^{H_K}. $$

\begin{thm}[{\cite[Prop. A 1.2.6]{fo4}}]
The functors $V$ and $\underline M$ induces an exact tensor equivalence between the categories $'\t{Mod}_{\O_{\E_K}} ^\varphi$  and $'\t{Rep}_{\Z_p} (H_K)$.
\end{thm}

\begin{defn} \label{defn phi gamma}
An \'etale \emph{$(\varphi, \Gamma)$-module} is a triple $(M, \varphi_M, \Gamma_M)$ where
\begin{itemize}
\item $(M , \varphi_M)$ is an \'etale $\varphi$-module;
\item $\Gamma_M$ is a continuous $\O_{\E_K}$-semi-linear $\Gamma_K$-action on $M$ which commutes with $\varphi_{M}$.
\end{itemize}
\end{defn}

\begin{convention}
Since we also use $M$ to denote an \'etale $\varphi$-module, we will now use $\hat M =(M, \varphi_M, \Gamma_M)$ to denote an \'etale $(\varphi, \Gamma)$-module. Clearly, this notation compares with the \'etale $(\varphi, \tau)$-modules that we use in \cite{GL}.
\end{convention}

For an \'etale $(\varphi, \Gamma)$-module $\hat M =(M, \varphi_M, \Gamma)$, define
$$ V(\hat M):= ( \O _{\widehat \E ^\ur} \otimes_{\O_{\E_K}} M) ^{\varphi =1},$$
which is a $G_K$-representation.
For $V $ in  $'\t{Rep}_{\Z_p}(G_K) $, define
$$ \underline{\hat{M}}(V):= ( \O _{\widehat \E ^\ur} \otimes_{\Z_p} V) ^{H_K},$$
which is an \'etale $(\varphi, \Gamma)$-module.

\begin{thm}[\cite{fo4}]
The functors $V$ and $\underline{\hat{M}}$ induces an exact tensor equivalence between the categories $'\t{Mod}_{\O_{\E_K}} ^{\varphi, \Gamma}$  and $'\t{Rep}_{\Z_p} (G_K)$.
\end{thm}

\subsection{Wach $\varphi$-modules and Wach modules}
In this subsection, we assume $K=K_0$. As we mentioned earlier, we have $\O_\E=\O_{\E_K}$ in this case.
We write
$$q:=\frac{\varphi( [\underline \varepsilon] -1)}{[\underline \varepsilon] -1} = \frac{\varphi(u)}{u} =\frac{(1+u)^p-1}{u}.$$
\begin{defn} For a nonnegative integer $r$,
we write $'\sfi$  for the category of finite-type $\gs$-modules $\M$ equipped with
a $\varphi_{\gs}$-semilinear endomorphism $\varphi_\M : \M \to \M$ satisfying
\begin{itemize}
	\item the cokernel of the linearization $1\otimes \varphi: \varphi ^*\M \to \M$ is killed by $q^r$;
	\item the natural map $\M \to \O_\E \otimes_{\gs} \M$ is injective.
\end{itemize}
 Morphisms in $'\sfi$ are $\varphi$-compatible $\gs$-module homomorphisms.
\end{defn}
We call objects in $'\sfi$ \emph{Wach $\varphi$-module of height $r$}. The category of {\em finite free Wach $\varphi$-modules of height $r$},
denoted $\sfi$, is the full subcategory of $'\sfi$ consisting
of those objects which are finite free over $\gs$. We call an object $\M \in {'\sfi}$ a {\em torsion Wach $\varphi$-module of height $r$} if $\M$ is killed by $p^n$ for some $n$.

For any Wach $\varphi$-module $\M \in '\sfi$, we define
$$T^\ast (\M):= (\M \otimes_{\gs } \O _{\widehat \E ^\ur} ) ^{\varphi=1}.$$


\begin{definition}
A finite free (resp. torsion)\emph{Wach module of height $r$} is a triple $\hat \M =(\M, \varphi_\M, \Gamma_\M)$, where
\begin{enumerate}
\item $(\M, \varphi_\M)\in {'\sfi}$ is a finite free (resp. torsion) Wach $\varphi$-module of height $ r$;
\item $\Gamma_\M$ is a continuous $\Gamma_K$-action on $\M$ which commutes with $\varphi_\M$;
\item $\Gamma_\M$ acts on $\M / u\M$ trivially.
 \end{enumerate}
\end{definition}
For any Wach module $\hat \M$, we can attach a $\Z_p [G_K]$-module via
$$\hat{T}^\ast (\hat \M):= (\M \otimes_{\gs } \O _{\widehat \E ^\ur} ) ^{\varphi=1}.$$


\begin{theorem} \cite{Ber, Kisin-Ren} \label{thm: WachBerger}
The functor $\hat{T}^\ast$ induces an anti-equivalence between
the category of finite free Wach modules of height $r$
and the category of $G_K$-stable $\Z_p$-lattices in crystalline representations with Hodge-Tate weights in $\{-r, \dots, 0\}$.
\end{theorem}

\begin{rem} \label{rem Wach citing}
The above theorem in this form is first proved by Berger \cite[Thm. 2]{Ber}, which critically uses results by Wach and Colmez \cite{Wac96, Col99}. However, the work in \cite{Col99} (which proves that crystalline representations are of finite height) uses the overconvergence theorem in \cite{Colmez-overcon}.

Fortunately, there is a second proof of the above theorem by Kisin and Ren (\cite[Thm. 0.1]{Kisin-Ren}), which is similar to \cite{kisin2} (and does not use results in \cite{Colmez-overcon}). Note that both \cite{kisin2} and \cite{Kisin-Ren} use Kedlaya's result on slope filtration over the Robba ring (\cite{Kedlaya}), but that is independent of \cite{Colmez-overcon}.

In summary, by citing \cite[Thm. 0.1]{Kisin-Ren} for Theorem \ref{thm: WachBerger}, there would be \emph{no circular reasoning} in our paper to reprove overconvergence theorem of \cite{Colmez-overcon}.
\end{rem}

\begin{convention}
From now on in this paper, by a crystalline representation, we always mean a crystalline representation with \emph{non-positive} Hodge-Tate weights.
\end{convention}

\subsubsection{} \label{tate twist}
For any integer  $s \ge 0$, let $\hat \M(\varepsilon_p ^{-s})$ denote the Wach $\varphi$-module corresponding to $\varepsilon_p^{-s}$ via our Theorem \ref{thm: WachBerger}.
It is easy to check that this is a rank-1 $\gs$-module with a base $\fe$ so that $\varphi (\fe) = q^s \fe $.  For any Wach $\varphi$-module $\M$, we denote $\M (s):  = \M \otimes \M (\varepsilon _p^{-s})$. So $\varphi_{\M({s})} = q^s \varphi_\M$.

\subsection{Wach $\varphi$-models and Wach models}
In this subsection, we assume $K=K_0$. As we mentioned earlier, we have $\O_\E=\O_{\E_K}$ in this case.
\begin{defn}
\begin{enumerate}
  \item Given an \'etale $\varphi$-module $M$ in  $'\t{Mod} _{\O_\E} ^{\varphi}$.   If $\M \in {'\sfi}$ is a Wach $\varphi$-module so that $M =  \O_\E \otimes_\gs \M $, then $\M$ is called a \emph{Wach $\varphi$-model} of $M$, or simply a \emph{model} of $M$.

  \item  Given $\hat M:  = (M, \varphi_M, \Gamma_M)$ a torsion (resp. finite free) $(\varphi, \Gamma)$-module. A torsion (resp. finite free) Wach module  $\hM : = (\M, \varphi_\M, \Gamma_\M)$ is called a \emph{model} of $\hat M$ if $\M$ is a model of $M$ and the isomorphism $\O_\E \otimes_{\gs} \M \simeq M $ is compatible with $\Gamma_K$-actions on both sides.
\end{enumerate}
\end{defn}

The following lemma is obvious.
\begin{lemma}
  \begin{enumerate}
    \item If $\M$ is a model of $M$, then $T^\ast(\M)\simeq V(M)$ as  $\Z_p[H_K]$-modules.
    \item If $\hM$ is a model of $\hat M$ then $\hat{T}^\ast(\hM)\simeq V(\hat M)$ as $\Z_p[G_K]$-modules.
  \end{enumerate}
\end{lemma}

Clearly, the most natural models of \'etale $(\varphi, \Gamma)$-modules come from lattices in crystalline representations.
\begin{lemma} \label{lem model}
Suppose $T$ is a $G_K$-stable $\Z_p$-lattice in a crystalline representation  of $G_K$ with Hodge-Tate weights in $\{-r, \dots, 0 \}$.
Let $\hat M:  = (M, \varphi_M, \Gamma_M)$ be the $(\varphi, \Gamma)$-module associated to $T$.
Let $\hM : = (\M, \varphi_\M, \Gamma_\M)$ be the Wach-module associted to $T$.
Then $\hM$ is a model of $\hat M$.
\end{lemma}
\begin{proof}
Easy.
\end{proof}

Now suppose that $T_n$ is a $p$-power torsion representation of $G_K$, and $M_n$ the \'etale $\varphi$-module associated to $T_n|_{H_K}$. A natural source of Wach $\varphi$-models of $M_n$ comes from the following. Suppose that we have a surjective map of $G_K$-representations $f: L \onto T_n$ where $L$ is a crystalline finite free $\Zp$-representation, then it induces the surjective map (which we still denote by $f$) $f: \mathcal L \onto M_n$, where $\mathcal L$ is the \'etale $\varphi$-module associated to $L|_{H_K}$. If $\mathfrak L$ is the Wach $\varphi$-module associated to $L$, then by Lemma \ref{lem model} above, $\mathfrak L$ is a Wach $\varphi$-model of $\mathcal L$. And so $f(\mathfrak L)$ is clearly a Wach $\varphi$-model of $M_n$.

\begin{example} \label{ex-cyclotomic} Let $T_1= \BF_p $ be the trivial $G_K$-representation and $M_1$ denote the corresponding trivial \'etale $\varphi$-module. Then $\M_1 = \ku$ is a model of $M_1$ which is realized by the surjection $\Z_p \twoheadrightarrow \BF_p $.
We also have another surjection $\Z_p (1-p) \twoheadrightarrow  \BF_p$, which realizes another model $\M_1 (p-1)\subset M_1$.
It is easy to check that $\M_1 (p-1) = u ^{p-1} \M_1$.
\end{example}

\subsubsection{Comparison between Kisin modules and Wach $\varphi$-modules} \label{item comparison}
Let us compare the following two situations:
\begin{enumerate}
  \item (Situation 1.) Kisin modules (see, e.g., \cite[\S 2]{GL} for more details) for a general $K$ where $[K: K_0]=e$ (and let us fix a uniformizer $\pi$ of $K$, with Eisenstein polynomial $E(u) \in W(k)[u]$ of degree $e$). Note that $\varphi(u) =u^p$ in this case.
  \item (Situation 2.) Wach $\varphi$-modules when $K=K_0$. Recall that $q =\frac{(1+u)^p-1}{u} \in W(k)[u]$ of degree $p-1$. Note that $\varphi(u) =(1+u)^p-1$ in this case.
\end{enumerate}
The theories in these two situations have lots of similarities.
In particular, when we prove something for Wach $\varphi$-modules, the proof could be almost verbatim as the corresponding result for Kisin modules; oftenly, we only need to change the $E(u)$ in Situation 1 to $q$ in Situation 2 (and some other slight modifications). We already see this in \S \ref{tate twist} (compare with \cite[Ex. 2.3.4]{GL}).
More notably, when we are in modulo $p$ situations, note that $E(u)=u^e (\bmod p)$ and $q =u^{p-1} (\bmod p)$; then oftenly, we only need to change $e$ in Situation 1 to $p-1$ in Situation 2 to prove results about Wach modules. We already see this in Example \ref{ex-cyclotomic} (compare with \cite[Ex. 2.4.4]{GL}).
Here is another illustration of the general phenomenon (switch between $e$ and $p-1$):
\begin{itemize}
\item Given a $p$-torsion Kisin module in (Situation 1) of $E(u)$-height $r$ (cf. \cite[\S 2]{GL}), let $A$ be the matrix for $\varphi$, then there exists $B$ such that $AB=u^{er}$.
\item Given a $p$-torsion Wach $\varphi$-module in (Situation 2) of $q$-height $r$, let $A$ be the matrix for $\varphi$, then there exists $B$ such that $AB=u^{(p-1)r}$.
\end{itemize}

\subsection{Overconvergence} \label{subsec OC}
In this subsection, we let $K$ be as in \S \ref{subsub K}.

For any $x \in W(\Fr R)$, we can write $x = \sum \limits_{i = 0}^\infty p ^i [x _i]$ with $x_i\in \Fr R$. Denote $v_R (\cdot)$ the valuation on $\FrR$ and normalized by
$$v_R(u (\bmod p))=v_R(1-\varepsilon) =\frac{p}{p-1}.$$
For any $r \in \mathbb R^{>0}$, set
$$ W(\Fr R) ^{\dag, r}: =\left \{ x= \sum _{i = 0}^\infty p ^i [x _i] \in W(\FrR) | \  i  +   r\frac{p-1}{p} v_R (x_i ) \to + \infty \right \}. $$
It turns out that $ W(\FrR) ^{\dag, r}$ is a ring, stable under $G_K$-action but not Frobenius, namely $\varphi (W(\FrR)^{\dag, r}) = W(\FrR)^{\dag, r/p}$. See \cite[\S II.1]{Colmez-overcon} for more details.
For any subring $A \subset W(\Fr R)$, denote
$$A^{\dag, r}: =A \cap W(\Fr R) ^{\dag, r}.$$

Recall that for a finite free $\Z_p$-representation $T$ of $G_K$, we can associate the $(\varphi, \Gamma)$-module via:
$$\underline{\hat{M}}(T): =  (\O _{\widehat \E ^\ur} \otimes_{\Z_p} T) ^{H_K}.$$
Now for any $r>0$, we define
$$ \underline{\hat{M}}^{\dagger, r} (T): =(\O _{\widehat \E ^\ur}^{\dagger, r} \otimes_{\Z_p} T) ^{H_K}.$$

\begin{definition}\label{Def-OC}
For a finite free $\Z_p$-representation $T$ of $G_K$, its associated $(\varphi, \Gamma)$-module is called overconvergent if there exists $r > 0$ such that
$$ \underline{\hat{M}}(T)= \O_\E  \otimes_{\O_\E ^{\dag, r}}   \underline{\hat{M}}^{\dagger, r} (T).$$
In particular, if above holds, we say that $\underline{\hat{M}}(T)$ is overconvergent on the interval $(0, r]$.
\end{definition}

The main theorem of \cite{Colmez-overcon} is the following:
\begin{thm}[\cite{Colmez-overcon}] \label{thm main OC}
For any finite free $\Z_p$-representation $T$ of $G_K$, its associated \'etale $(\varphi, \Gamma)$-module is overconvergent.
\end{thm}

The following lemma says that we can reduce the proof of Theorem \ref{thm main OC} to the case when $K=K_0$.

\begin{lemma} \label{lem unram}
To prove Theorem \ref{thm main OC}, it suffices to prove the cases when $K$ is unramified.
\end{lemma}
\begin{proof}
Suppose we have already proved Theorem \ref{thm main OC} when the base field is unramified.
Now let $K$ be any field with $[K:K_0]=e$.
Given a finite free $\Z_p$-representation $T$ of $G_K$ of rank $d$, then the induction $\Ind_{G_K}^{G_{K_0}} T$ is a finite free $\Z_p$-representation of $G_{K_0}$ of rank $ed$.
By assumption, $\Ind_{G_K}^{G_{K_0}} T$ is overconvergent as a $G_{K_0}$-representation. By \cite[Thm. II. 3.2.(i)]{Colmez-overcon} $\Res_{G_K}^{G_{K_0}} \Ind_{G_K}^{G_{K_0}} T$ is overconvergent as a $G_{K}$-representation.
By Mackey decomposition, $\Res_{G_K}^{G_{K_0}} \Ind_{G_K}^{G_{K_0}} T$ contains $T$ as a direct summand, and so $T$ has to be overconvergent as a $G_K$-representation too.
Note that if $\Ind_{G_K}^{G_{K_0}} T$ is overconvergent on the interval $(0, r]$ as a $G_{K_0}$-representation, then $T$ is also overconvergent on the interval $(0, r]$ as a $G_{K}$-representation
\end{proof}

\subsection{A reproof when $K/\Qp$ is finite extension} \label{subsec reproof}
The main result of our paper is:
\begin{thm} \label{thm ram OC}
Suppose $K/\Qp$ is a finite extension, and let $f: =[k: \mathbb F_p]$.
Let $T$ be a finite free $\Zp$-representation of $G_K$ of rank $d$.
Then $\underline{\hat{M}}(T)$ is overconvergent on the interval
\begin{equation} \label{eqeq radius}
(0, \frac{1}{21pfe^2d^2p^{fed}}].
\end{equation}
\end{thm}

By Lemma \ref{lem unram}, it suffices to prove Thm. \ref{thm ram OC} when $K=K_0$.
Then the main input will be Theorem \ref{thm OC basis} (which is analogue of \cite[Thm. 6.1.1]{GL}), which says that there exists an ``overconvergent basis" of $\underline{\hat{M}}(T)$, with respect to which all entries of matrices for $\varphi$  and $\gamma \in \Gamma_K$ are overconvergent elements.

\begin{remark}
\begin{enumerate}
  \item \label{item1} As we can see from \eqref{eqeq radius}, we have obtained a ``uniform" lower bound of overconvergence radius for \emph{all} $p$-adic Galois representations, once we fix $p, e, f, d$.
It is very possible that this uniform bound can have applications to situations where we consider a ``family" of Galois representations (where $p, e, f, d$ are naturally fixed).

\item However, let us also mention that it seems almost impossible to use methods in this paper to reprove the overconvergence property of the family of $(\varphi, \Gamma)$-modules attached to a family of Galois representations (as in the setting of \cite{BC08}); indeed, it seems impossible to construct and study loose crystalline liftings of a family of residual representations. (This does not contradict with the potential usefulness mentioned in Item \ref{item1} above.)

  \item Actually, we only need to prove the $K=\Qp$ case to deduce the full Theorem \ref{thm ram OC}. So for example, we can consider inductions of the form $\Ind_{G_K}^{G_{\Qp}} T$. This will not save much trouble for us (besides, we still use some finite unramified extensions in e.g. Lemma \ref{lem-max-descent}).

\end{enumerate}

\end{remark}


\section{Torsion Wach models} \label{sec Wach}
In this section, we always assume $K/\Qp$ is a finite extension and $K=K_0$.
We study liftable Wach $\varphi$-models in torsion \'etale $\varphi$-modules.

\subsection{Maximal Wach models and devissage}
For $n \in \mathbb Z^{>0}$, suppose $T_n$ is a $p^n$-torsion representation of $G_K$ ($T_n$ is not necessarily free over $\mathbb Z/p^n\mathbb Z$). Let $M_n$ be the \'etale $\varphi$-module corresponding to $T_n |_{H_K}$. Recall that in the torsion case, a Wach $\varphi$-module $\M \subset M_n $ is called a \emph{Wach $\varphi$-model} if $\M[\frac 1 u] = M_n$.

\begin{defn}
  A Wach $\varphi$-model $\M$ is called \emph{loosely liftable} (in short, liftable), if there exists a $G_{K}$-stable $\Z_p$-lattice $L$ inside a crystalline representation and surjective map $f : L \sto T_n$ such that the corresponding map of \'etale $\varphi$-modules $f : \mathcal L \to  M_n$  satisfies $f(\L)=\M$, where $\mathcal L$ and $\L$ are the \'etale $\varphi$-modules and Wach modules for $L$ (see the discussion before Example \ref{ex-cyclotomic}).
In this case, we say that $\M$ can be \emph{realized} by the surjection $f : L \sto T_n$.
\end{defn}

\begin{defn}\label{defbig}
Let $E$ be a finite extension of $\Q_p$, $\O_E$ its ring of integers, $\varpi_E$ a uniformizer, $\mathfrak m_E$ the maximal ideal and $k_E= \O_E/ \mathfrak m_E$ the residue field.
\begin{enumerate}
  \item For $\bar \rho: G_K \to \GL_d(k_E)$ a torsion representation, we say that a continuous representation $r : G_K \to \GL _d (\O_E)$ is a \emph{strict} $\O_E$-lift of $\bar \rho$ if $r(\bmod \mathfrak m_E) \simeq \bar \rho$.

\item For a $p$-torsion representation of the form $\bar \rho: G_K \to \GL_d(k_E)$, we say that ``$E$ is big enough" (for $\bar\rho$) if each direct summand of $\bar \rho^{\textnormal{ss}}$ has strict crystalline $\O_E$-lift with non-positive Hodge-Tate weights.
\end{enumerate}
\end{defn}

\begin{thm}\label{thm loose lifts}
Suppose $T_n$ is a $p^n$-torsion representation of $G_K$ ($T_n$ is not necessarily free over $\mathbb Z/p^n\mathbb Z$), then it admits loose crystalline lifts, i.e., there exists a $G_{K}$-stable $\Z_p$-lattice $L$ inside a crystalline representation and surjective map $f : L \sto T_n$.
Furthermore, if we let $T_1:=T_n/pT_n$ and suppose $E$ is big enough for $T_1\otimes_{\Fp}k_E$, then we can always make the loose crystalline lift to be finite free over $\O_E$.
\end{thm}
\begin{proof}
This is \cite[Thm 3.3.2]{GL}; let us sketch the proof here. We hope it can serve as a quick guide for readers not familiar with the proof of \emph{loc. cit.}; also, we will use many of the (intermediate) results later in \S \ref{subsec h}.

\textbf{Step 1: $n=1$ case.} Let $E$ be the degree $d$ unramified extension of $K_0$, where $d=\dim_{\Fp} T_1$; let $\mathcal{O}_E$ be the ring of integers and let $k_E$ be the residue field.
\begin{itemize}[leftmargin=0cm]
\item By \cite[Lem. 3.1.3]{GL}, $E$ is  big enough for $T_1\otimes_{\Fp} k_E$ (cf. Def. \ref{defbig}).


\item Then \cite[Thm. 3.2.1]{GL} shows that there exists a finite unramified extension ${K'}$ of $K$, such that the restriction of $T_1\otimes_{\Fp} k_E$ to ${G_{{K'}}}$      admits a strict crystalline $\O_E$-lift $\rho'$. Then $L =\Ind^{G_K}_{G_{{K'}}} \rho'$ is a loose crystalline lift of $T_1$.
\end{itemize}

\textbf{Step 2: induction argument.}
Now suppose our theorem is true for $n-1$, and consider the case $n$.
\begin{itemize}[leftmargin=0cm]
\item  Denote $T_{n-1}:=T_n/p^{n-1}T_n$. By induction hypothesis, let $f_{n-1} : \widetilde L_{n-1} \onto T_{n-1}$ be a loose crystalline lift.
Let $\breve W_{n}$ be the cartesian product of $\widetilde L_{n-1} \onto T_{n-1}$ and $T_n \onto T_{n-1}$. We have the following diagram of short exact sequences (of $\Zp[G_K]$-modules).
\begin{equation}\label{Diag-0}
\begin{split}
\xymatrix{ 0 \ar[r] & p^{n-1} T_{n}\ar[d] \ar[r] & \breve W_{n} \ar@{->>}[d]\ar[r] & \widetilde L_{n -1}\ar@{->>}[d] \ar[r] & 0 \\
0 \ar[r] & p^{n-1} T_{n} \ar[r] & T_{n} \ar[r] & T_{n-1} \ar[r] & 0}
\end{split}
\end{equation}
Let $Z_{n}: = \breve W_{n}/ p \breve W_{n}$, then one readily checks that the following is  short exact:
\begin{equation}\label{Eq-exactsq-Z-n}
 0 \to p ^{n-1} T_n \to Z_{n} \to \widetilde L_{n-1}/ p \widetilde L_{n-1} \to 0.
\end{equation}

\item Let $E$ be the ``big enough" field for the $p$-torsion representation $ T_n/pT_n$ as in Step 1. 
Since $T_n/pT_n \onto p ^{n -1} T_{n}$, $E$ is also big enough for $p ^{n -1} T_{n}\otimes_{\Fp}k_E$.
By Step 1, there exists $K''/K$ finite unramified such that there exists a strict crystalline $\O_E$-lift
$\widetilde N_{n}''  \onto p ^{n -1} T_{n}\otimes_{\Fp}k_E|_{G_{K''}} $. Also, $\widetilde L_{n-1}$ is a strict crystalline $\O_E$-lift of $\widetilde L_{n-1}/\varpi_E\widetilde L_{n-1}$.
 Via (Statement B) in the proof of \cite[Thm. 3.2.1]{GL} (which constructs extensions of strict crystalline lifts), we can find some finite unramified extension $K'/K''$ such that we have the following diagram of short exact sequences of $G_{K'}$-representations:
\begin{equation}\label{Diag-2 first apperance}
\begin{split}
\xymatrix{0 \ar[r] &  \widetilde N_{n}' (:=\widetilde N_{n}''|_{G_{K'}})   \ar@{->>}[d] \ar[r] &   \widetilde Z_{n}'  \ar@{->>}[d]\ar[r] & \widetilde L_{n-1}(-s)  \ar@{->>}[d]\ar[r]  & 0 \\
0 \ar[r] & p ^{n -1} T_{n}\otimes_{\Fp}k_E \ar@{->>}[d]\ar[r] & Z_{n}\otimes_{\Fp}k_E    \ar@{->>}[d]\ar[r] & \widetilde L_{n-1}/ p\widetilde L_{n-1}\otimes_{\Fp}k_E   \ar@{->>}[d]\ar[r] & 0\\
0 \ar[r] & p ^{n -1} T_{n}\ar[r] & Z_{n}   \ar[r] & \widetilde L_{n-1}/ p\widetilde L_{n-1}  \ar[r] & 0  ,}
\end{split}
\end{equation}
where the first row are strict crystalline $\O_E$-lifts of the second row, and the maps from the second row to the third row are compatible projections to chosen $\Fp$-direct summands. Here we can easily compute that we can choose
\begin{equation}\label{eqnss}
s:= p^{fd}+p-2
\end{equation}
where $d=\dim_{\Fp} T_n/pT_n$.

\item  Define $ \widetilde L_{n}': =  \widetilde Z_{n}' \times_{Z_n} \breve W_{n} $, where $\widetilde Z_{n}' \to {Z_n}$ comes from diagram \eqref{Diag-2 first apperance}.
The $G_{K'}$-representation $ \widetilde L_{n}'$ sits in the following diagram of short exact sequences:
\begin{equation}\label{Diag-1-NEW}
\begin{split}
\xymatrix{ 0 \ar[r] &  p \widetilde L_{n-1}|_{G_{K'}}\ar[d] \ar[r] & \widetilde L_{n}' \ar@{->>}[d]\ar[r] & \widetilde Z_{n}'\ar@{->>}[d] \ar[r] & 0 \\   0 \ar[r] & p \widetilde L_{n-1}|_{G_{K'}} \ar[r] & \breve W_{n}|_{G_{K'}} \ar[r] & Z_{n}|_{G_{K'}} \ar[r] & 0}
\end{split}
\end{equation}
One easily sees that $\widetilde L_{n}'$ is $\mathcal O_E$-finite free and crystalline. Furthermore $\widetilde L_{n}'$ maps surjectively onto $\breve W_n|_{G_{K'}}$, and so onto $T_n |_{G_{K'}}$. Finally, $\widetilde L_{n}:= \Ind_{G_{K'}}^{G_K} \widetilde L_{n}' $ is the desired loose crystalline lift of $T_n$.
\end{itemize}

\end{proof}

\begin{lemma} \label{lem-liftexist} \begin{enumerate}
\item For any $M_n$ (coming from $T_n |_{H_K}$), a liftable model $\M \subset M_n$ exists.
\item The set of liftable models inside $M_n$ has a unique maximal object (with respect to the obvious inclusion relation). (We denote the maximal liftable model as $\M_{(n)}$).
\end{enumerate}
\end{lemma}
\begin{proof}
 Item (1) follows from Thm. \ref{thm loose lifts}.
Item (2) follows from similar argument as in \cite[Lem. 4.1.2]{GL}; let us sketch the argument for the reader's convenience. Denote the set of all liftable models inside $M_n$ as $\textnormal{LF}_{\gs}^{\infty}(M_n)$. It is trivial to see that $\textnormal{LF}_{\gs}^{\infty}(M_n)$ admits finite supremum. Then it suffices to show that there is an upper bound for length of any chain in $\textnormal{LF}_{\gs}^{\infty}(M_n)$; this is the content of  \cite[Lem. 3.2.4]{liu-car1}, which we need to reprove in our Wach $\varphi$-module setting. The proof is verbatim, if we change all $E(u)$ (resp. $e$) in \textit{loc. cit.} to $q$ (resp. $p-1$) (cf. \S \ref{item comparison}).
\end{proof}

\subsubsection{} \label{notation free}
Now let us introduce some notations.
Let $T$ be a $\Zp$-finite free $G_K$-representation, and set $T_n : = T/p ^n T$.
Let $M$ be the finite free \'etale $\varphi$-module corresponding to $T$, and set $M_n:=M/p^n M$.
Denote the natural projection $q_{j ,i } : M_j \onto  M_i$ for $i < j$ induced by modulo $p^i$.
Recall that we use $\M_{(j)}$ to denote the maximal liftable Wach $\varphi$-model of $M_j$.
Set $\M _{(j , i ) } = q_{j ,i} (\M_{(j)})$.

For $i<j$, we denote $\iota_{i, j}: M_i \to M_j$ the injective map where for $x \in M_i$, we choose any lift $\hat x \in M_j$, and let $\iota_{i, j}(x) =p^{j-i} \hat x$. This is clearly well-defined, and we will use it to identify $M_i$ with $\iota_{i, j}(M_i)= M_j[p^i]=p^{j-i}M_j$ (recall that the notation $M_j[p^i]$ denotes the $p^i$-torsion elements). The maps $\iota_{i, j}$ are clearly transitive; namely, $\iota_{i, j}\circ \iota_{j, k}=\iota_{i, k}$. Also, the composite
$$ M_i \overset{\iota_{i, j}}{\longrightarrow}  M_j \overset{q_{j, i}}{\longrightarrow}   M_i $$ is precisely the map $ M_i \overset{\times p^{j-i}}{\longrightarrow} M_i$.

Let $\M \in {'\sfi}$ be a torsion Wach $\varphi$-module such that $\M[\frac 1 u]=M_n$.
Following the discussion above \cite[Lem. 4.2.4]{liu2}, for each $0 \le i < j \le n$, we define
$$ \M ^{i, j}:= \t{Ker}  (p ^i \M \overset{p ^{j-i}}{\longrightarrow} p ^ j \M ).$$

\begin{lemma} \label{lem list}
We use the above notations. In particular, for $i<j$, we identify $M_i$ with $M_j[p^i]$. Then we have:
\begin{enumerate}
  \item $\m_{(j)}[p^i] = \m_{(i)}$ as Wach $\varphi$-models of $M_i$.
  \item We have $ \m_{(j)}^{i-1, i} = \m_{(i, 1)}  $ for $j\ge i$.
\end{enumerate}
\end{lemma}
\begin{proof}
The proof is strictly verbatim as in \cite[Lem. 4.1.5, Cor. 4.1.6]{GL}; let us sketch the main ideas. Suppose $f: L \onto T_j$ realizes $\m_{(j)}$. Let $g: L \onto T_j \onto p^iT_j$ be the composite map where the second map is the $\times p^i$ map, and let $K: =\Ker g$. Then the induced loose crystalline lifting $K \onto T_j[p^i]$ realizes $\m_{(j)}[p^i]$, and thus $\m_{(j)}[p^i] \subset \m_{(i)}$. For the other direction, note that $\m_{(i)}+\m_{(j)}$ is a liftable model in $M_j$. So we have $\m_{(i)}\subset \m_{(j)}$, and so $\m_{(i)}\subset \m_{(j)}[p^i]$.

For Item (2), note $ \m_{(j)}^{i-1, i}  = p^{i-1}(\m_{(j)}[p^i])$, which equals to $p^{i-1}\m_{(i)}$ (=$ \m_{(i, 1)}$) by Item (1).
\end{proof}

\subsubsection{} \label{subsubsec notation descent}
Suppose $T_n$ is a $p^n$-torsion representation of $G_K$, $M_n$ the corresponding \'etale $\varphi$-module.
Suppose $K'$ is a finite unramified extension of $K$, with residue field $k'$ and ring of integers $W(k')$.
Then it is easy to see that $M'_n  := W(k ') \otimes _{W(k)} M_n $ is the corresponding \'etale $\varphi$-module for $T_n|_{G_{K'}}$.

\begin{lemma}\label{lem-max-descent}
Use notations as in \ref{subsubsec notation descent}, and use $\M'_{(n)}$ to denote the maximal liftable model of $M'_n$.
Then we have
$\M'_{(n)} \simeq W(k') \otimes_{W(k)} \M_{(n)}$.
\end{lemma}
\begin{proof}
This is the $(\varphi, \Gamma)$-analogue of \cite[Lem. 4.2.10]{GL}.
Similarly as the beginning of \cite[\S 4.2]{GL}, we need to set $K'_{p^\infty}: =K'K_{p^\infty}$. Since we have $K'\cap K_{p^\infty}=K$ (because $K/\Qp$ is unramified), we still have $G_{p^\infty}/G_{p^\infty}' \simeq \Gal(K'/K)$, where $G_{p^\infty}':=\Gal(\overline{K}/K'_{p^\infty})$. Then all the argument in \cite[\S 4.2]{GL} carry over verbatim.
\end{proof}

\subsection{Existence of $h$}\label{subsec h}
\begin{prop}\label{prop-key}
Use notations in \ref{notation free}. Then there exists a constant $h$ only depending on $p$, $f$ and $d$ such that $u ^h \M_{(n-1, 1)} \subset \M_{(n ,1)}$ for all $n\geq 1$. Consequently $u ^h \M_{(n)} ^{i -1, i} \subset  \M_{(n)} ^{i , i +1}$ for all $i= 1, \dots , n -1$ by Lemma \ref{lem list}(2).
\end{prop}

\begin{proof}
The proof is basically the same as that for \cite[Prop. 5.2.1]{GL}, except a few minor changes (mainly due to the difference of Frobenius actions, cf. \S \ref{item comparison}). For the reader's convenience, we give a sketch of the main arguments; in particular, we point out the changes in our situation.

By Lemma \ref{lem-max-descent}, for any $n$, if we let $M_n':=W(k')\otimes_{W(k)} M_n$ (where $k'/k$ is any finite extension), then the maximal liftable Kisin model of $M_n'$ is $\M'_{(n)} \simeq W(k') \otimes_{W(k)} \M_{(n)}$. Thus, it is easy to see that $\M'_{(n, 1)}=W(k')\otimes_{W(k)} \M_{(n, 1)} $ for any $n$. So to prove our proposition, it suffices to show that
$$u ^h \M'_{(n-1, 1)} \subset \M'_{(n ,1)}.$$

We divide the following argument into two steps. In Step 1, we will construct another Kisin model (denoted as $\mathring{\M}'_{n, 1}$) of $M_1'$ such that $\mathring{\M}'_{n, 1} \subset \M'_{(n ,1)}$. Then in Step 2, we show that $u ^h \M'_{(n-1, 1)} \subset \mathring{\M}'_{n, 1}.$

\textbf{Step 1.} The loose crystalline lift $\widetilde L'_n \onto \breve W_n|_{G_{K'}} \onto  T_n|_{G_{K'}}$ realizes a liftable Kisin model $\mathring{\M}'_n$ of $M'_n$, and so $\mathring{\M}'_n \subset \M_{(n)}'$.
The following composite of $G_{K'}$-representations
\begin{equation} \label{easy 1}
\xymatrix{
\widetilde L'_n\ar@{->>}[r] & \breve W_n|_{G_{K'}} \ar@{->>}[r] &T_n|_{G_{K'}} \ar@{->>}[r] & T_1|_{G_{K'}}
.}
\end{equation}
realizes a Kisin model $\mathring{\M}'_{n, 1}$ in $M_1'$, and we certainly have
$$\mathring{\M}'_{n, 1} \subset \M'_{(n, 1)}.$$

By some elementary diagram chasing (using \eqref{Diag-0},  \eqref{Diag-1-NEW} and \eqref{Diag-2 first apperance}; cf. \cite[Prop. 5.2.1]{GL} for the chasing), the composite \eqref{easy 1} is the same as
\begin{equation}\label{Eq-sequence-m'}
 \xymatrix{\widetilde L'_n\ar@{->>}[r] &  \widetilde Z_n' \ar@{->>}[r] & \widetilde L_{n -1}(-s)  \ar@{->>}[r] &\widetilde L_{n -1}/p\widetilde L_{n -1}  \ar@{->>}[r] & T_1.}
\end{equation}
And so \eqref{Eq-sequence-m'} also realizes $\mathring{\M}'_{n, 1}$.

\textbf{Step 2.} The last surjection of \eqref{Eq-sequence-m'} is induced by the following composite of $G_{K}$-representations:
\begin{equation}\label{Eq-sequence-mn-1}
\xymatrix{\widetilde L_{n-1}\ar@{->>}[r] & T_{n-1} \ar@{->>}[r] & T_1}.
\end{equation}
We may assume that $\wt L_{n-1}\twoheadrightarrow T_{n-1}$ realize $\M_{(n -1)}$; thus the composite \eqref{Eq-sequence-mn-1} restricted to $G_{K'}$ realizes $\M'_{(n-1, 1)}$.

So in order to prove $u ^h \M'_{(n-1, 1)} \subset \M'_{(n ,1)}$, it suffices to show
\begin{equation}
u ^h \M'_{(n-1, 1)} \subset \mathring{\M}'_{n, 1}.
\end{equation}
And so it suffices to show that the cokernel of the following composite (which are maps of Wach modules corresponding to \eqref{Eq-sequence-m'}) is killed by $u^h$
\begin{equation}\label{Eq-uh}
\xymatrix{\widetilde{\L}'_{n} \ar[r] &  \widetilde{\mathfrak{Z}}_n' \ar[r] & \widetilde \L'_{n -1}(s)     \ar[r] & \L'_{n-1}/p\L'_{n-1}}.
\end{equation}
By the Wach module analogue of \cite[Lem. 5.2.2]{GL} (the proof in the Wach module setting is strictly \emph{verbatim}), the map $\widetilde{\L}'_{n} \to  \widetilde{\mathfrak{Z}}_n'$ is in fact surjective, so we only need to consider the cokernel of the following composite of maps:
\begin{equation}\label{Eq-uhuh}
\xymatrix{\widetilde{\mathfrak{Z}}_n' \ar[r] & \widetilde \L'_{n -1}(s)     \ar[r] & \L'_{n-1}/p\L'_{n-1}}.
\end{equation}
Denote the image of the composite \eqref{Eq-uhuh} as $\mathring \L'_{n-1}$, which is contained in $\L'_{n-1}(s)/p\L'_{n-1}(s)$.
So we can choose basis $e_1, \ldots e_m$ of $\L'_{n-1}$, such that $\mathring \L'_{n-1}$ has a $k'[\![u ]\!]$-basis formed by $u^{a_1}\bar e_1, \ldots u^{a_m}\bar e_m$, where
\begin{equation} \label{eqaibi}
a_i=s+b_i \text{ with } b_i\ge 0.
  \end{equation}
Here, the expression of ``$a_i$" in \eqref{eqaibi} is different from that in \cite{GL} (below \cite[Eqn. (5.2.5)]{GL}). This is essentially because in \emph{loc. cit.}, \cite[Ex. 2.4.4]{GL} is (implicitly) used; whereas in our situation, we use our corresponding Ex. \ref{ex-cyclotomic}.

Finally, it suffices to bound $a_i$; the proof follows similar ideas as in \cite[Lem. 5.2.6]{GL}.

\textbf{Convention}:
In the remaining of the proof, \emph{all} the representations that we consider are $G_{K'}$-representations, and all Wach modules are over $\gs'=W(k')\otimes_{W(k)} \gs$. To be completely rigorous, we will need to restrict many representations from $G_K$ to $G_{K'}$, and use prime notation over Wach modules (\emph{i.e.}, notations like $\M'$). For notational simplicity, we will drop these prime notations.

\begin{itemize}[leftmargin=0cm]
\item Firstly, given any $p$-torsion Wach $\varphi$-module $\M$ (over $k'[\![u]\!]$), denote (similarly as in \cite[Def. 5.2.4]{GL})
$$\alpha(\M): = v_R (\det (\varphi)).$$
The invariant $\alpha(\M)$ is well-defined, and is additive with respect to short exact sequences (as in \cite[Lem. 5.2.5(1)]{GL}). Furthermore, if $L$ is a $G_{K'}$-stable $\Z_p$-lattice in a crystalline representation $V$  with non-positive Hodge-Tate weights $\t{HT} (V)$, and $\L$ the corresponding Wach module, then by the proof of \cite[III. 3.1]{Ber}, we have
\begin{equation} \label{eq alpha}
 \alpha(\L/p\L) = (p-1)(\sum_{i \in \t{HT} (V)} -i).
\end{equation}
Note that \eqref{eq alpha} is the Wach module analogue of \cite[Lem. 5.2.5(2)]{GL}.

\item Now, recall that $\mathring \L_{n-1}$ has a $k'[\![u ]\!]$-basis formed by $u^{a_1}\bar e_1, \ldots u^{a_m}\bar e_m$, where $\bar e_1, \ldots \bar e_m$ is a $k'[\![u ]\!]$-basis of $\L_{n-1}/p\L_{n-1}$.
We clearly have
$$\alpha(\mathring \L_{n-1}) = (p-1)\sum_{1 \le i\le m}a_i +\alpha (\widetilde \L_{n-1}/p \widetilde \L_{n-1})
=(p-1)\sum_{1 \le i\le m}b_i +\alpha \left(\widetilde \L_{n-1}(-s)/p \widetilde \L_{n-1}(-s)\right).$$
 Since $b_i \le \sum_{1 \le i\le m}b_i$, so we only need to bound $\alpha(\mathring \L_{n-1})-\alpha (\widetilde \L_{n-1}(-s)/p \widetilde \L_{n-1}(-s))$.
By exactly the same (easy) argument as in \cite[Lem. 5.2.6]{GL} (which uses \cite[Lem. 5.2.2]{GL}), we have
\begin{eqnarray*}
 \alpha(\mathring \L_{n-1})-\alpha \left(\widetilde \L_{n-1}(-s)/p \widetilde \L_{n-1}(-s)\right)
 &\le & \alpha(\widetilde{\mathfrak Z_n}/p\widetilde{\mathfrak Z_n}) - \alpha \left(\widetilde \L_{n-1}(-s)/p \widetilde \L_{n-1}(-s)\right) \\
 &= & (p-1) ( \sum_{i \in \t{HT} (\widetilde N_n)} -i  ), \textnormal{ by additivity of } \alpha \text{ and \eqref{eq alpha}}.
\end{eqnarray*}
Note in the current paper, the $\Qp$-dimension of  $\widetilde{N}_n[\frac 1 p]$ (constructed in the first row of \eqref{Diag-2 first apperance}) should be $fd^2$, and so
$$(p-1)\sum_{1 \le i\le m}b_i  \leq  (p-1)fd^2(p^{fd}-2). $$
And so we can choose any $h$ such that $h \leq s + fd^2(p^{fd}-2)$.
So in our current paper, we can choose
\begin{equation}\label{eq h}
  h: = 3fd^2p^{fd}
\end{equation}

\end{itemize}

\end{proof}

\section{Overconvergent basis and main theorem} \label{sec OC}
In this section, we prove our main theorem; then we also make some comparison with known proofs.

\subsection{Overconvergent basis}
We first show the existence of an ``overconvergent basis", with respect to which all entries of the matrices for $\varphi$ and $\gamma \in \Gamma_K$ are overconvergent elements. In the following, we will use notations like $(e_j)$ to mean a row vector $(e_j)_{j=1}^d=(e_1, \cdots, e_d)$.

\begin{theorem}\label{thm OC basis}
Suppose $K/\Qp$ is finite extension and $K=K_0$.
Let $T$ be a finite free $\Zp$-representation of $G_K$ of rank $d$. Let $h = 3fd^2p^{fd}$ be as in \eqref{eq h}.
Then there exists an $\O_\E$-basis $(e_1, \dots, e_d)$ of $\underline{\hat{M}}(T)$ such that,
\begin{itemize}
  \item $\varphi(e_j) =(e_j) A$ with $A\in \Md(\gs [\![\frac{p}{x}]\!]) \subset \Md( \gs [\![\frac{p}{u^{4h+3ph}}]\!]   )  $ where $x=u^{4h}(\varphi(u)^{3h})$.

    \item  $\gamma(e_j) =(e_j) B$ with $B \in \Md(\gs [\![\frac{p}{y}]\!])$ where $y=y_\gamma=(\varphi(u))^{4h}  (\gamma( u))^{3h}$, for any $\gamma \in \Gamma_K$.
\end{itemize}
\end{theorem}
\begin{proof}
The proof follows exactly the same ideas  as in \cite[Thm. 6.1.1]{GL}.
Indeed, the construction of the basis $(e_1, \dots, e_d)$ is \emph{verbatim} as Step 1, Step 2 and Step 3 of \emph{loc. cit.} (which we sketch below for the reader's convenience). We only need to modify Step 4 of \textit{loc. cit.} in our situation.

\textbf{Step 1.} \emph{Generators of $\m_{(n)}$}.
First of all, by induction on $n$, we construct a specific set of generators $\{\e_{(n), j}^{(i)},  1 \le i \le n\}_{j=1}^d$ of $\m_{(n)}$ such that $\{p^{i-1}\e_{(n), j}^{(i)}\}_{j=1}^d$ forms a $\ku$-basis of $\m_{(n)}^{i-1, i}$.

We choose $\{\e_{(1), j}^{(1)}\}_{j=1}^d$ any $\ku$-basis of $\m_{(1)}$.
Suppose we have defined $\{\e_{(n-1), j}^{(i)},  1 \le i \le n-1\}_{j=1}^d$ the generators for $\m_{(n-1)}$.
By Lem. \ref{lem list}(2), we can use the map $\iota_{n-1, n}: M_{n-1} \to M_{n}$ to identify $\m_{(n-1)}^{i-1, i}$ with $\m_{(n)}^{i-1, i}$ when $1 \le i \le n-1$. Now define
\begin{itemize}
  \item $\e_{(n), j}^{(i)} : =\iota_{n-1, n} (\e_{(n-1), j}^{(i)}) $ for $1 \le i \le n-1, 1\le j \le d$, and
  \item choose any $\{\e_{(n), j}^{(n)}\}_{j=1}^d$ in $\m_{(n)}$, so that $(p^{n-1}\e_{(n), j}^{(n)})$ is
  a $\ku$-basis of $\m_{(n)}^{n-1, n}=p^{n-1}\m_{(n)}$.
\end{itemize}
This finishes the inductive definition.

\textbf{Step 2.} \emph{Basis of $M_n$.}
With above, now we define basis for $M_n$.
By \cite[Lem. 5.1.2(2)]{GL} (note that this lemma is about \emph{module theory} over $\gs$; there is no $\varphi$ involved), for any $x\ge 1$, we can write
\begin{eqnarray} \label{eq1}
(p\e_{(x), j}^{(x)}) =(\iota_{x-1, x}(\e_{(x-1), j}^{(x-1)}))  \Lambda_{x-1} (I_d+ \frac{p}{u^{2h}} Z_{x-1})
\end{eqnarray}
with $\Lambda_{x-1} \in \Md(\gs)$ such that $u^h\Lambda_{x-1}^{-1} \in \Md(\gs)$ and $Z_{x-1} \in \Md(\gs[\frac{p}{u^{2h}}])$.
Let
$$(e^{(n)}_j)   =  (\e_{(n), j}^{(n)}) Y_n^{-1},\quad \text{ where } Y_n:=\prod_{x=1}^{n-1}  (  \Lambda_{x-1} ( I_d +  \frac{p}{u^{2h}} Z_{x-1}  ) ).$$
We can arrange that $Y_n \in \Md(\frac{1}{u^{2(n-1)h}}\gs)$, and $Y_n^{-1} \in \Md(\frac{1}{u^{3(n-1)h}}\gs)$ (because everything is $p^n$-torsion here), and we can easily check (by Nakayama Lemma) that  $ (e^{(n)}_1, \ldots e^{(n)}_d)$ is an $\mathcal O_{\mathcal E, n}$-basis of $M_n$.

\textbf{Step 3.} \emph{Compatibility of basis.}
We can easily check that $(e^{(n-1)}_j) = (e^{(n)}_j  \mod p^{n-1})$.
So now we can define $(e_j) : =\lim_{n \to \infty} (e^{(n)}_j)$, which is a basis of $M$.

\textbf{Step 4.} \emph{Matrices for $\varphi$ and $\gamma$.} Now, to prove our theorem, similarly as in Step 4 of \cite[Thm. 6.1.1]{GL} (which used \cite[Lem. 6.1.3]{GL}), we can apply our Lemma \ref{lem shrink ring}(2), and so it suffices to show that (for any $\gamma \in \Gamma_K$),
\begin{itemize}
  \item $\varphi(e_j^{(n)}) =(e_j^{(n)}) A_n$ with $A_n\in \Md(\frac{1}{x^{(n-1)}}\gs)$ where $x=u^{4h}(\varphi(u)^{3h})$.

  \item $\gamma(e_j^{(n)}) =(e_j^{(n)}) B_n$ with $B_n \in \Md(\frac{1}{y^{(n-1)}}\gs)$ where $y =(\varphi(u))^{4h}  (\gamma( u))^{3h}$.
\end{itemize}
Since $\m_{(n)}$ comes from a loose crystalline lift, we can write
\begin{itemize}
  \item $\varphi(\e_{(n), j}^{(n)}) = \sum_{i=1}^n (\e_{(n), j}^{(i)}) P_i^{(1)}$, with $P_i^{(1)} \in \Md(\gs)$,
   \item $\gamma(\e_{(n), j}^{(n)}) = \sum_{i=1}^n (\e_{(n), j}^{(i)}) Q_i^{(1)}$, with $Q_i^{(1)} \in \Md(\gs)$.
\end{itemize}
By \cite[Lem. 5.1.2]{GL} (note that this lemma is about \emph{module theory} over $\gs$; there is no $\varphi$ involved), for all $i<n$, we can write
$$ (\e_{(n), j}^{(i)}) =(\e_{(n), j}^{(n)}) Y_{i, n}  \text{ with }  Y_{i, n} \in \Md(\gs[\frac{p}{u^{2h}}]).$$
So we can write
\begin{itemize}
  \item $\varphi(\e_{(n), j}^{(n)}) =(\e_{(n), j}^{(n)}) P_n$, with $P_n \in \Md(\gs[\frac{p}{u^{2h}}])$,
  \item  $\gamma(\e_{(n), j}^{(n)}) =(\e_{(n), j}^{(n)}) Q_n$, with $Q_n \in \Md(\gs[\frac{p}{u^{2h}}])$.
\end{itemize}
Then we have
\begin{itemize}
\item $A_n=  Y_n P_n \varphi(Y_n^{-1})$,
 \item $B_n=  Y_n  Q_n \gamma(Y_n^{-1})$,
\end{itemize}
and we can easily conclude.
Finally, the containment $ \Md(\gs [\![\frac{p}{x}]\!]) \subset \Md( \gs [\![\frac{p}{u^{4h+3ph}}]\!]   ) $ is via Lemma \ref{lem gs frac}(1).
\end{proof}

\begin{lemma} \label{lem shrink ring}
Suppose $v \in W(R)-pW(R)$.
\begin{enumerate}
\item $W(R)[\![\frac{p}{v}]\!] \subset W(\Fr R)^\dagger$.
\item Suppose that $y_n = \frac{x_n}{v^{(n-1)}} \in W_n(\Fr R)$ such that  $y_{n +1} \equiv y _{n}  \mod p ^n $ in $W_{n+1} (\Fr R)$.
If $x_n \in \gs$ for all $n$, then $y_n$ converges to a $y \in \gs[\![\frac{p}{v}]\!]$.
\end{enumerate}
\end{lemma}
\begin{proof}
Item (1) is easy. Item (2) is similar to  \cite[Lem. 6.1.3]{GL}, which is a special case of the current lemma (with $v=u^\alpha$).
\end{proof}

\begin{lemma} \label{lem gs frac}
Suppose $a \in \mathbb Z^{\ge 0}, b, m \in \mathbb Z^{\ge 1}$, then we have
\begin{enumerate}
  \item  $\gs   [\![  \frac{p}{ u^a \varphi(u^b) }       ]\!]   \subset  \gs   [\![  \frac{p}{ u^{a+bp}  }       ]\!] $.

  \item $p^m \varphi(u^{mb}) = p^m u^{pmb} (\bmod p^{m+1}\gs)$.

  \item  $\varphi(u^{mb})  \in  u^{pmb  -pm} W(R) (\bmod  p^{m+1}W(\Fr R)).  $

\end{enumerate}
\end{lemma}
\begin{proof}
Write $\varphi(u) =u^p+pc$ for some $c \in \gs$. Then for $x = \sum_{n \ge 0} s_n (\frac{p}{ u^a \varphi(u^b) } )^n \in  \gs   [\![  \frac{p}{ u^a \varphi(u^b) }       ]\!]$ where $s_n \in \gs$, we have
$$ x = \sum_{n\ge 0}  s_n (\frac{p}{ u^{a+bp}  } )^n  \cdot \frac{1}{ ( 1+\frac{pc}{u^{p}}  )^{bn}   }   . $$
Since $\frac{1}{ ( 1+\frac{pc}{u^{p}}  )^{bn}   }  \in \gs [\![  \frac{p}{ u^{p}  }       ]\!]  \subset   \gs   [\![  \frac{p}{ u^{a+pb} }]\!] $, we can deduce (1). For (2) and (3), simply consider the expansion of $(u^p+pc)^{mb}$.
\end{proof}

\subsection{Proof of Main theorem}
First, we prove our main theorem when $K=K_0$.
\begin{prop} \label{thm final OC}
Suppose $K/\Qp$ is a finite extension and $K=K_0$.
Let $T$ be a finite free $\Zp$-representation of $G_K$ of rank $d$.
Then $\underline{\hat{M}}(T)$ is overconvergent on the interval
\begin{equation}
(0, \frac{1}{21pfd^2p^{fd}}].
\end{equation}
\end{prop}
\begin{proof}
The proof is similar to that of \cite[Thm. 6.2.1]{GL}.

Let $(e_1, \dots, e_d)$ be the basis of $M$ as in Theorem \ref{thm OC basis}. Let $(t_1, \dots, t_d)$ be any basis of $T$, and let $(e_1, \dots, e_d)=(t_1, \dots, t_d)X$ where $X \in \Md(\O_{\widehat \E^\ur})$. Then we have $\varphi(X)=XA$.
In order to prove the theorem, it suffices to show that $X \in \Md(W(R)[\![\frac{p}{u^\beta}]\!])$ for
\begin{equation}
\beta =4h+3ph+2 \leq 7ph = 21pfd^2p^{fd}
\end{equation}


To prove the above assertion, by Lemma \ref{lem shrink ring}, it suffices to show that $u^{(n-1)\beta}X_n \in \Md(W_n(R))$ where $X_n:=X(\mod p^n)$. We prove this by induction on $n$.
The case $n=1$ is verbatim as in \cite[Thm. 6.2.1]{GL}.
Suppose the claim is true when $n \leq m$, and let us consider the case $n=m+1$.
Write
$X = \sum\limits_{\ell = 0} ^\infty p ^\ell X'_\ell$ with $ X'_\ell \in \Md([\Fr R])$ where $[\Fr R]$ is the set of Teichm\"{u}ller lifts,
so that we have $X_{m+1}=X_m +p^m X'_m (\bmod p^{m+1})$.
It suffices to show $u^{m\beta} X'_m \in \Md([R])$.

From $\varphi(X_{m+1})=X_{m+1}A_{m+1} (\bmod p^{m+1}W(\Fr R))$, we have
\begin{equation} \label{eq11}
\varphi(X_m) + p^m \varphi(X_m') =X_mA_{m+1} + p^m X_m'A_{m+1} (\bmod p^{m+1}W(\Fr R))).
\end{equation}
Multiply both sides of \eqref{eq11} with $\varphi(u^{m\beta})$. Then we have
\begin{itemize}[leftmargin=*]
  \item $\varphi(u^{m\beta})\varphi(X_m) \in \Md(W(R))$ by induction hypothesis.

  \item The term $$\varphi(u^{m\beta}) X_mA_{m+1} = \frac{\varphi(u^{m\beta})}{u^{(m-1)\beta + m(\beta -2)}} u^{(m-1)\beta}X_m  u^{m(\beta-2)}A_{m+1} ,$$
      where
     \begin{itemize}
       \item $\frac{\varphi(u^{m\beta})}{u^{(m-1)\beta + m(\beta -2)}} \in W(R) (\bmod p^{m+1}W(\Fr R)))$ by Lemma \ref{lem gs frac}(3),
       \item $ u^{(m-1)\beta}X_m \in \Md(W(R))$ by induction hypothesis,
       \item and $u^{m(\beta -2)} A_{m+1} \in \Md(W(R))$ by Theorem \ref{thm OC basis}.
     \end{itemize}
      So the term $\varphi(u^{m\beta}) X_mA_{m+1} \in \Md(W(R))  (\bmod p^{m+1}W(\Fr R)))$.

  \item  The term $p^m \varphi(u^{m\beta}) X_m'A_{m+1} =p^m u^{m\beta} X_m'  u^{pm\beta -m\beta} A_{m+1} (\bmod p^{m+1}W(\Fr R)))$  by Lemma \ref{lem gs frac}(2), and we have  $u^{pm\beta -m\beta} A_{m+1} \in \Md(W(R))$.
\end{itemize}
So in the end, we get an equation of the form
\begin{equation}\label{eq12}
  p^m \varphi(Y) =p^m YB+C (\bmod p^{m+1}W(\Fr R)),
\end{equation}
where $Y =u^{m\beta} X'_m, B, C \in \Md(W(R))$. We must have $C\in p^m\Md(W(R))$, and so we can divide \eqref{eq12} by $p^m$ and apply \cite[Lem. 6.2.2]{GL} to conclude.

\end{proof}

We can deduce the full case from above:
\begin{thm} \label{thm final OC for K}
Suppose $K/\Qp$ is a finite extension.
Let $T$ be a finite free $\Zp$-representation of $G_K$ of rank $d$.
Then $\underline{\hat{M}}(T)$ is overconvergent on the interval
\begin{equation} \label{eq r}
(0, \frac{1}{21pf(ed)^2p^{fed}}].
\end{equation}
\end{thm}
\begin{proof}
This is easy corollary of Theorem \ref{thm final OC}, via Lemma \ref{lem unram}. The expression of the overconvergence interval \eqref{eq r} is because the $\Zp$-rank of the induction $\Ind_{G_K}^{G_{K_0}} T$ is $ed$. Indeed, \eqref{eq r} is an overconvergence interval for $\Ind_{G_K}^{G_{K_0}} T$ (as a $G_{K_0}$-representation), hence also for $T$ (as a $G_K$-representation).
\end{proof}

\subsection{Comparison with known proofs} \label{sub compa}
It is natural to wonder how to compare our proof with the classical proof of \cite{Colmez-overcon} (and also \cite{BC08, Kednew}). However, there does not seem to be any obvious link between the two proofs. One particular characteristic of our proof is that the key technical analysis happen over the \emph{imperfect} period rings (such that $\mathfrak S$ and $\mathcal O_{\mathcal E}$); whereas the classical proofs \cite{Colmez-overcon, BC08} (also implicitly in \cite{Kednew}) rely on a Tate-Sen formalism (cf. \cite[Def. 3.1.3]{BC08}) for certain \emph{perfect} period rings (cf., e.g., \cite[Prop. 4.2.1]{BC08}). The link between the two worlds remain mysterious.

As we already mentioned in the introduction, the explicit uniform bound on the overconvergence radius \eqref{eq r} is a new result.
Actually, there are some results concerning the overconvergence radius in \cite[\S 4.2]{BC08} (we thank an anonymous referee for pointing out this reference). In the following, we show that using \textit{loc. cit.}, together with results from \cite{GL} (on loose crystalline lifting) and \cite{liu-car2} (on ramification bound), we can also prove a certain (implicit) ``uniform overconvergence".

In the following, let $K/\Qp$ be a finite extension. Let $T$ be a finite free $\Zp$-representation of $G_K$ of rank $d$. Suppose $L/K$ is a finite Galois extension such that $G_L$ acts trivially on $T/12pT$ (so \emph{a priori}, $L$ depends on $T$, or at least on $T/12pT$).

Let $n(L)=n(G_L)>0$ be the integer defined as in \cite[Def. 3.1.3(TS3)]{BC08}, and let $s(L/K)>0$ be the number defined in \cite[Lem. 4.2.5]{BC08}. It is easy to see that both $n(L)$ and $s(L/K)$ depend only on $L$ and $K$.
Now, let $r(L, K) = \min \{ \frac{1}{(p-1)p^{n(L)-1}}, \frac{1}{s(L/K)}   \}$.
Then by \cite[Prop. 4.2.6]{BC08}, $\underline{\hat{M}}(T)$ is overconvergent on the interval $(0, r(L, K)]$ (note that the convention of $W(\Fr R)^{\dagger, r}$ in \cite[\S 4.2]{BC08} is different from ours in \S \ref{subsec OC}).

\begin{lemma} \label{lem L}
Suppose $p>2$. Then there exists a finite Galois extension $L=L(K, p, d)$ over $K$ which depends only on $K, p, d$, such that $G_L$ acts on $T/12pT$ trivially.
\end{lemma}

\begin{rem}\label{rem comparison}
By above lemma, when $p>2$, all the numbers $n(L)$, $s(L/K)$ and $r(L, K)$ depend only on  $K, p, d$ (and not on $T$). So the $\underline{\hat{M}}(T)$'s are overconvergent ``uniformly" (depending on $K, p, d$ only).
However, in this method, we do not know how to explicitly bound the overconvergence radius, because it seems difficult to explicitly bound $n(L)$ and $s(L/K)$ (even though we can bound $L$ quite explicitly, see the proof in the following).
\end{rem}

\begin{proof}[Proof of Lemma \ref{lem L}]
It suffices to show the existence of $L$ such that $G_L$ acts on $T/p^2T$ trivially.
By \cite[Thm. 3.3.2, Rem. 3.3.5]{GL}, there exists a loose crystalline lift of $T/p^2T$, with Hodge-Tate weights in the range $[-2(p^{fd}+p-2), 0]$. Now we can apply \cite[Thm. 1.1]{liu-car2} to conclude. Namely, in \emph{loc. cit.}, we simply let $n=2$ and $r=2(p^{fd}+p-2)$. Clearly, the $s$ and $\mu$, and thus $G_s^{\mu}$ in \emph{loc. cit.} depend on $K, p, d$ only. We then simply let our $L$ to be the Galois closure of the fixed field of $G_s^{\mu}$.
\end{proof}

\bibliographystyle{amsalpha}

\end{document}